\documentclass[11pt]{article}

\usepackage{cite}
\usepackage{amsmath,amsthm,amssymb,amsfonts}
\usepackage{fullpage}

\newtheorem{theorem}{Theorem}

\newtheorem{lemma}[theorem]{Lemma}
\newtheorem{conjecture}{Conjecture}

\title{A proof of Brouwer's toughness conjecture}

\author{Xiaofeng Gu\thanks{University of West Georgia, Carrollton, GA 30118 ({\tt xgu@westga.edu}).
This research was partially supported by a grant from the Simons Foundation (522728, XG).}
}

\begin{document}
\date{}
\maketitle
\noindent

\begin{abstract}
The toughness $t(G)$ of a connected graph $G$ is defined as $t(G)=\min\{\frac{|S|}{c(G-S)}\}$,
in which the minimum is taken over all proper subsets $S\subset V(G)$ such that $c(G-S)>1$,
where $c(G-S)$ denotes the number of components of $G-S$.
Let $\lambda$ denote the second largest absolute eigenvalue of the adjacency matrix of a graph.
For any connected $d$-regular graph $G$, it has been shown by Alon that
$t(G)>\frac{1}{3}(\frac{d^2}{d\lambda+\lambda^2}-1)$, through which,
Alon was able to show that for every $t$ and $g$ there are $t$-tough graphs of girth strictly greater than $g$,
and thus disproved in a strong sense a conjecture of Chv\'atal on pancyclicity.
Brouwer independently discovered a better bound $t(G)>\frac{d}{\lambda}-2$ for any connected $d$-regular graph $G$, 
while he also conjectured that the lower bound can be improved to $t(G)\ge \frac{d}{\lambda} - 1$. 
We confirm this conjecture.
\end{abstract}



\section{The conjecture}
We use $\lambda_i(G)$ to denote the $i$th largest eigenvalue of the adjacency matrix of a simple graph $G$ on $n$ vertices,
for $i=1,2,\cdots, n$. By the Perron-Frobenius Theorem, $\lambda_1$ is always positive (unless $G$ has no edges) and 
$|\lambda_i|\le \lambda_1$ for all $i\ge 2$. Let $\lambda = \max_{2\le i\le n} |\lambda_i| =\max\{|\lambda_2|, |\lambda_n|\}$, 
that is, $\lambda$ is the second largest absolute eigenvalue. For a $d$-regular graph, it is well known that $\lambda_1=d$.

Let $c(G)$ denote the number of components of a graph $G$.
The {\bf toughness} $t(G)$ of a connected graph $G$ is defined as $t(G)=\min\{\frac{|S|}{c(G-S)}\}$,
where the minimum is taken over all proper subsets $S\subset V(G)$ such that $c(G-S)>1$.
A graph $G$ is {\bf $k$-tough} if $t(G)\ge k$. 
Graph toughness was introduced by Chv\'atal \cite{Chva73} in 1973 and is closely related to 
many graph properties, including connectivity, Hamiltonicity, pancyclicity, factors, spanning trees, etc.

The study of toughness from eigenvalues was initiated by Alon~\cite{Alon95} who showed that
for any connected $d$-regular graph $G$, $t(G)>\frac{1}{3}(\frac{d^2}{d\lambda+\lambda^2}-1)$, through which,
Alon was able to show that for every $t$ and $g$ there are $t$-tough graphs of girth strictly greater than $g$.
This strengthened a result of Bauer, Van den Heuvel and Schmeichel \cite{BaVS95} who showed the same for $g=3$, and 
thus disproved in a strong sense a conjecture of Chv\'atal \cite{Chva73} that there exists a constant $t_0$ such that 
every $t_0$-tough graph is pancyclic.

Brouwer~\cite{Brou95} independently discovered a better bound and showed that $t(G)>\frac{d}{\lambda}-2$ for a connected 
$d$-regular graph $G$. He mentioned in~\cite{Brou95} that the bound might be able to be improved a little to $\frac{d}{\lambda}-1$,
and then posed the exact conjecture as an open problem in \cite{Brou96}. 
Some partial results have been provided in \cite{CiGu16}. However, no substantial progress has been made for 
more than two decades. Most recently, the author improved the bound to $t(G) > \frac{d}{\lambda} - \sqrt{2}$ in \cite{Gu21}.

\begin{conjecture}[Brouwer~\cite{Brou95, Brou96}]
\label{tconj}
For any connected $d$-regular graph $G$, $t(G)\ge \frac{d}{\lambda}-1$.
\end{conjecture}

It is mentioned by Brouwer~\cite{Brou95} that there are infinitely many graphs $G$ such that $t(G)\le d/\lambda$
with equality in many cases (for example, strongly regular graphs constructed in  \cite{Brou95} and \cite{CiWo14},
and many Kneser graphs \cite{POHBWWC}). Brouwer~\cite{Brou95} also pointed out that Conjecture~\ref{tconj}, if true, is tight. 
To see sharpness, one may notice that the toughness can be arbitrarily close to $0$, while $d/\lambda\ge 1$ holds for all 
$d$-regular graphs. Cioab\u{a} and Wong~\cite{CiWo14} briefly described an explicit construction of such extremal graphs.

In this paper, we confirm Conjecture~\ref{tconj}.
\begin{theorem}
\label{tofg}
Let $G$ be a connected $d$-regular graph. Then $\displaystyle t(G)\ge \frac{d}{\lambda}-1.$
\end{theorem}

\section{The proof}
Our main tool is the expander mixing lemma. A $d$-regular graph on $n$ vertices with the second largest absolute eigenvalue 
at most $\lambda$ is called an $(n, d, \lambda)$-graph. 
It is well known that an $(n, d, \lambda)$-graph  for which $\lambda = \Theta(\sqrt{d})$ is a very good pseudo-random graph 
behaving, in many aspects, like a truly random graph $G(n, p)$. The quantitative definition of pseudo-random graphs was introduced 
by Thomason~\cite{Thom85, Thom87} who defined jumbled graphs.

The celebrated expander mixing lemma is usually attributed to Alon and Chung~\cite{AlCh88}. 
This idea appeared earlier with a different form in the PhD thesis~\cite{Haem79} of Haemers.
We present this lemma  in Theorem~\ref{expmix} (see \cite{KrSu06} for a complete proof).
A variation can be found in \cite[Chapter 9]{AlSp16} by Alon and Spencer.
We refer readers to the informative survey \cite{KrSu06} by Krivelevich and Sudakov for more about the expander mixing lemma 
and pseudo-random graphs. As mentioned in \cite{KrSu06}, the expander mixing lemma is a truly remarkable result, connecting 
edge distribution and graph spectrum, and providing a very good quantitative handle for the uniformity of edge distribution based 
on graph eigenvalues.

For every two subsets $A$ and $B$ of $V(G)$, let $e(A, B)$ denote the number of edges with
one end in $A$ and the other one in $B$ (edges with both ends in $A\cap B$ are counted twice).
We use $e(A)$ to denote the number of edges with both ends in $A$, and thus $e(A, A) = 2 e(A)$.
\begin{theorem}[Expander Mixing Lemma]
\label{expmix}
Let $G$ be an $(n, d, \lambda)$-graph. Then for every two subsets $A$ and $B$ of $V(G)$, 
\begin{equation}
\label{genmix}
\left| e(A, B) - \frac{d}{n} |A| |B| \right| \le \lambda \sqrt{|A| |B| \left(1-\frac{|A|}{n}\right)\left(1-\frac{|B|}{n}\right)}.
\end{equation}
In particular, 
\begin{equation}
\label{oneset}
\left| e(A) - \frac{d}{2n} |A|^2 \right| \le \frac{\lambda}{2} |A| \left(1-\frac{|A|}{n}\right).
\end{equation}
\end{theorem}

The following lemma has been proved and used by Brouwer and Haemers~\cite{BrHa05}. For the sake of completeness,
we include a short proof here.
\begin{lemma}[\hspace{-1sp}\cite{BrHa05}]
\label{indexlem}
Let $x_1,\cdots, x_c$ be positive integers such that $\sum_{i=1}^c x_i \le 2c-1$. Then for every integer $\ell$ with
$0\le \ell \le \sum_{i=1}^c x_i$, there exists an $I\subset \{1,\cdots, c\}$ such that $\sum_{i\in I} x_i =\ell$.
\end{lemma}
\begin{proof}
The proof provided by Brouwer and Haemers~\cite{BrHa05} goes by induction on $c$. The case $c=1$ is trivial.
Let $c\ge 2$ and assume that $x_1\le \cdots \le x_c$. Suppose it is true for $c-1$ integers.
Let $\ell' = \ell$ if $\ell \le c-1$ and $\ell' = \ell - x_c$ otherwise.
Since $c-1\le \sum_{i=1}^{c-1} x_i = \sum_{i=1}^c x_i - x_c\le 2(c-1)-1$, by inductive hypothesis,
there exists an $I'\subset \{1,\cdots, c-1\}$ such that $\sum_{i\in I'} x_i =\ell'$.
Let $I = I'$ or $I = I'\cup \{c\}$, as desired.
\end{proof}

Now we are ready to present the proof of the main theorem.
\begin{proof}[\bf Proof of Theorem~\ref{tofg}]
Suppose to the contrary that 
\begin{equation}
\label{contr}
t(G) < \frac{d}{\lambda}-1.
\end{equation}
By definition, suppose $S$ is a subset of $V(G)$ such that $\frac{|S|}{c(G-S)}=t(G)$. Let $B = V(G-S)$.
Denote $|V(G)|=n$, $c(G-S)=c$ and $t(G)=t$. Then $|S| = tc$ and so $|B| = n - tc$.

First we show that
\begin{equation}
\label{cupb}
c\le \frac{\lambda n}{d + \lambda}.
\end{equation}
In fact, (\ref{cupb}) can be obtained from the well-known Hoffman ratio bound. Here we give a direct proof.
To see this, let $U$ be a set of vertices that consists of exactly one vertex from each component of $G-S$. 
Then $|U|=c$ and $e(U)=0$. By (\ref{oneset}), $\left| e(U) - \frac{d}{2n} |U|^2 \right| \le \frac{\lambda}{2} |U| \left(1-\frac{|U|}{n}\right)$,
and thus $\frac{d}{n} |U| \le \lambda \left(1-\frac{|U|}{n}\right)$, which implies that $c=|U|\le  \frac{\lambda n}{d + \lambda}$.\\

If $|B| \le \frac{2\lambda n}{d+ \lambda}$, then $tc = n - |B| \ge \frac{(d - \lambda)n}{d+\lambda}$. Together with (\ref{cupb}), we have
$$t= (n - |B|)/c \ge \frac{(d - \lambda)n}{d+\lambda} \cdot \frac{d + \lambda}{\lambda n} 
= \frac{d - \lambda}{\lambda} = \frac{d}{\lambda} -1,$$
contrary to (\ref{contr}). Thus, we may assume that 
\begin{equation}
\label{B-lowb}
|B| > \frac{2\lambda n}{d+ \lambda},
\end{equation}
By (\ref{cupb}) and (\ref{B-lowb}), we have $|B| > 2c$, that is
\begin{equation}
\label{bvsc}
|B|\ge 2c+1.
\end{equation}

Let $V_1,\cdots, V_c$ denote the vertex sets of the $c$ components of $G-S$.
Without loss of generality, we may assume that $|V_1|\le \cdots\le |V_c|$.

\par\medskip
\noindent
{\bf Claim 1:} $\sum_{i=1}^{c-1} |V_i|\ge c$.
\\Proof of Claim 1: Otherwise, $\sum_{i=1}^{c-1} |V_i| = c-1$, that is, each $V_i$ is a single vertex for $i=1,\cdots, c-1$. 
Let $A= \cup_{i=1}^{c-1} V_i$ and thus $|A|=c-1$ and $e(A, B) = 0$. By (\ref{genmix}),
$$\frac{d}{n} |A| |B| \le \lambda \sqrt{|A| |B| \left(1-\frac{|A|}{n}\right)\left(1-\frac{|B|}{n}\right)}
< \lambda \sqrt{|A| |B|\left(1-\frac{|B|}{n}\right)},$$
and so
$$\left(\frac{d}{n} |A| |B| \right)^2 < \lambda^2 |A| |B|\left(1-\frac{|B|}{n}\right),$$
which implies that
$$\frac{d^2 |A| |B|}{\lambda^2 n} < n-|B| = tc.$$

Thus 
$$t > \frac{|A|}{c}\cdot \frac{d^2}{\lambda^2}\cdot \frac{|B|}{n} = \frac{c-1}{c}\cdot \frac{d^2}{\lambda^2}\cdot \frac{|B|}{n}
= \left(1-\frac{1}{c} \right) \cdot \frac{d^2}{\lambda^2}\cdot \frac{|B|}{n}\ge \frac{d^2}{2\lambda^2}\cdot \frac{|B|}{n}.$$

By (\ref{B-lowb}), $\frac{|B|}{n} > \frac{2\lambda}{d+ \lambda}$ and we have
$$t> \frac{d^2}{2\lambda^2}\cdot \frac{2\lambda}{d+ \lambda} = \frac{d^2}{\lambda(d+ \lambda)}
=\frac{d^2 - \lambda^2 + \lambda^2}{\lambda(d+ \lambda)} 
= \frac{d^2 -\lambda^2}{\lambda(d+ \lambda)} + \frac{\lambda^2}{\lambda(d+ \lambda)}
> \frac{d}{\lambda} -1,$$ contrary to (\ref{contr}). This completes the proof of Claim 1.

\par\medskip
\noindent
{\bf Claim 2:} $V_1,\cdots, V_c$ can be partitioned into two disjoint sets $X$ and $Y$ such that $e(X, Y)=0$, $|X|\ge c$ and  $|Y|\ge c$.
\\Proof of Claim 2:
If $|V_c|\ge c$, then  let $X = \cup_{i=1}^{c-1} V_i$ and $Y = B-X = V_c$. Then $|Y|\ge c$ and by Claim 1, $|X|\ge c$.
Thus we may assume that $|V_c| \le c-1$. We also notice that $|V_c|\ge 3$, for otherwise, $|B|=\sum_{i=1}^c |V_i| \le 2c$, contradicting
(\ref{bvsc}). Thus $3\le |V_c| \le c-1$. Let $\ell = c - |V_c|$. Then $1\le \ell \le c-3$.

If $\sum_{i=1}^{c-1} |V_i|\le 2(c-1) -1 = 2c-3$, then by Lemma~\ref{indexlem}, there exists an $I\subset \{1,\cdots, c-1\}$ such that 
$\sum_{i\in I} |V_i| =\ell$. Let $X = \cup_{i\in I} V_i \cup V_c$ and $Y = B-X$. Clearly $|X| = \ell + |V_c| =c$ and $|Y| = |B| -|X|\ge c+1$.

If $\sum_{i=1}^{c-1} |V_i| > 2(c-1) -1 = 2c-3$, then let $V'_i$ be a nonempty subset of $V_i$ for each $i=1,2,\cdots c-1$ such that 
$\sum_{i=1}^{c-1} |V'_i| = 2c-3$. We can easily do it by removing vertices from some $V_i$'s one by one until the total number of 
remaining vertices is exactly $2c-3$ (keep at least one vertex in each subset). By Lemma~\ref{indexlem}, there exists an 
$I\subset \{1,\cdots, c-1\}$ such that $\sum_{i\in I} |V'_i| =\ell$. Then $\sum_{i\not\in I, i< c} |V'_i| = (2c-3)-\ell \ge (2c-3) - (c-3) = c$.
Let $X = \cup_{i\in I} V_i \cup V_c$ and $Y = B-X = \cup_{i\not\in I, i< c} V_i$. 
Clearly $|X| \ge \ell + |V_c| =c$ and $|Y| = \sum_{i\not\in I, i< c} |V_i|\ge \sum_{i\not\in I, i< c} |V'_i| \ge c$.
This completes the proof of Claim 2.

\par\medskip
By Claim 2, $V_1,\cdots, V_c$ can be partitioned into two disjoint sets $X$ and $Y$ such that $e(X, Y)=0$, $|X|\ge c$ and  $|Y|\ge c$.

Since $e(X, Y)=0$, by (\ref{genmix}),
$\frac{d}{n} |X| |Y| \le \lambda \sqrt{|X| |Y| \left(1-\frac{|X|}{n}\right)\left(1-\frac{|Y|}{n}\right)},$
and so $$ d^2 |X|^2 |Y|^2 \le \lambda^2 |X| |Y| \left(n- |X|\right)\left(n-|Y|\right),$$
which implies that 
\begin{equation}
\label{xynoeg}
|X| |Y| \le \frac{\lambda^2}{d^2} (n - |X|)(n - |Y|).
\end{equation}

Without loss of generality, we may assume that $|X|\le |Y|$. By (\ref{xynoeg}), we have
$$ |X|^2 \le |X|\cdot |Y| \le \frac{\lambda^2}{d^2}(n-|X|)(n-|Y|)\le \frac{\lambda^2}{d^2}(n-|X|)^2,$$
that is $$|X| \le \frac{\lambda}{d} (n - |X|),$$
and hence 
\begin{equation}
\label{xxbound}
|X| \le \frac{\lambda n}{d+\lambda}.
\end{equation}

Also, since $|Y| = n - |S| - |X|$, by (\ref{xynoeg}) again, we have
$$ |X| (n - |S| - |X|) = |X|\cdot |Y| \le \frac{\lambda^2}{d^2}(n - |X|)(n - |Y|) = \frac{\lambda^2}{d^2}(n - |X|)(|S| + |X|),$$
which implies that
\begin{equation}
\label{xnxn}
|X|n \le \left(\frac{\lambda^2}{d^2}(n - |X|) + |X| \right) \left(|S| + |X| \right)
= \left(\frac{\lambda^2}{d^2}n + \frac{d^2 - \lambda^2}{d^2}|X| \right) \left(|S| + |X| \right).
\end{equation}
By (\ref{xxbound}), we have
$$\frac{d^2 - \lambda^2}{d^2}|X| \le \frac{d^2 - \lambda^2}{d^2}\cdot \frac{\lambda n}{d+\lambda} = \frac{d\lambda - \lambda^2}{d^2}n,$$
plugging in (\ref{xnxn}),
$$|X|n \le \left(\frac{\lambda^2}{d^2}n + \frac{d\lambda - \lambda^2}{d^2}n \right) \left(|S| + |X| \right) = \frac{\lambda n}{d} \left(|S| + |X| \right),$$
and we have
$$ |X| \le \frac{\lambda}{d} \left(|S| + |X| \right).$$
Hence 
\begin{equation*}
tc = |S| \ge \left(\frac{d}{\lambda} -1 \right) |X|,
\end{equation*}
implying that
$$t \ge \left(\frac{d}{\lambda} -1 \right)\cdot \frac{|X|}{c} \ge \frac{d}{\lambda} -1,$$
completing the proof of the theorem.
\end{proof}

\end{document}